\documentclass{amsart} 
\usepackage{amsmath, amssymb, cases, color, amscd}
\numberwithin{equation}{section}
\usepackage{paralist}
\usepackage{stmaryrd}
\usepackage{amsfonts}
\usepackage{mathrsfs}
\usepackage{graphicx}
\usepackage{yhmath}
\usepackage{upref}
\usepackage{mathtools}
\usepackage[misc]{ifsym}
\usepackage{epsfig} 
\usepackage{epstopdf} 
\usepackage[colorlinks=true]{hyperref}
\hypersetup{urlcolor=blue, citecolor=red}
\allowdisplaybreaks

\newtheorem{theorem}{Theorem}[section]
\newtheorem{corollary}[theorem]{Corollary}

\newtheorem{lemma}[theorem]{Lemma}
\newtheorem{proposition}[theorem]{Proposition}

\theoremstyle{definition}
\newtheorem{definition}[theorem]{Definition}
\newtheorem{remark}[theorem]{Remark}

\newcommand{\rfb}[1]{\mbox{\rm
		(\ref{#1})}\ifx\undefined\stillediting\else:\fbox{$#1$}\fi}

\renewcommand{\geq} {\geqslant}

\newcommand{\mm}    {{\hbox{\hskip 0.5pt}}}

\newcommand{\bluff} {{\hbox{\raise 15pt \hbox{\mm}}}}
\newcommand{\sbluff}{{\hbox{\raise 10pt \hbox{\mm}}}}

\title[Small oscillations of a floating cylinder]
{On an initial value problem describing the small oscillations of a floating cylinder} 

\author[Vicente Ocqueteau and Marius Tucsnak]{}


\thanks{$^*$Corresponding author}

\begin{document}
\maketitle

\centerline{\scshape
Vicente Ocqueteau$^{1*}$
and Marius Tucsnak$^1$}

\medskip

{\footnotesize
 \centerline{$^1$Institut de Math\'ematiques de Bordeaux (IMB)}
 \centerline{Universit\'e de Bordeaux}
 \centerline{351, Cours de la Lib\'eration - F 33 405, France}

} 

\medskip


\bigskip


\begin{abstract}
We study a coupled PDE-ODE system modeling the small oscillations of a floating cylinder interacting with small water waves. 
We consider the case when the floating body is an infinite circular cylinder, so that the equations of the free surface
of the fluid can be written in one space dimension.
The governing equations are formulated as an abstract evolution equation in a suitable Hilbert space, and we establish the well-posedness of the associated initial value problem. A key element of the proof is the analysis of a partial Dirichlet-to-Neumann map on an unbounded domain with a non-smooth boundary.
\end{abstract}


\mathtoolsset{showonlyrefs}
\section{Introduction} 
Understanding the dynamics of a floating body interacting with ocean waves is a problem of major interest for many application fields, such as naval engineering, floating wind turbines or  wave energy converters.
The mathematical analysis of the models describing such interactions is a complex problem. Indeed, in addition to the difficulties raised by the equations of water waves, new challenges are added by the presence of the fluid-solid interface and to the coupling of the equations describing the various components of the system. In particular, the possible appearance of corners in the fluid domain induces difficulties which are not usually present in the study of water waves in  smooth (possibly periodic) domains. We refer to the recent work of Lannes and Ming \cite{lannes2024posednessfjohnsfloating} for a more detailed description of the state of the art in this field.  When the full water system is replaced by some reduced models valid within the shallow water approximation or the Boussinesq regime, several results are available in the literature, see, for instance, Lannes and Iguchi \cite{iguchi2021hyperbolic},  Bresch, Lannes and M{\'e}tivier \cite{bresch2021waves} and the references therein.

 Most of the existing works on wave-structure interactions are based on the linear potential flow approach, which has been introduced in a rigorous manner in John \cite{john1949motion} and then used in various articles. These works consider essentially time harmonic motions, for both rigid or deformable bodies (see, for instance, Kuznetsov and Mazya \cite{kuznetsov2003linear} or Hazard and Lenoir \cite{hazard1993determination}). Much less has been done on the analysis of the corresponding systems in the time domain. As far as we know, within the linear potential flow approach, 
 the well-posedness of the initial value problem describing the coupled motion of the floating object and of the free surface of the fluid has not been explicitly considered in the literature. Some of the difficulties raised by this problem, namely those connected to the simultaneous presence of corners and of an unbounded fluid domain,  have been recently tackled in \cite{lannes2024posednessfjohnsfloating}. In the above mentioned reference, the authors consider  the linear potential flow approach, when the horizontal dimension is $d=1$, the fluid is supposed of finite depth and the floating solid is replaced by a partially immersed obstacle.
 
Our work tackles a much simpler geometric context, introduced in Ursell \cite{ursell_1964_decay}.  The considered system  describes the coupled vertical oscillations of a floating infinite circular cylinder and of the free surface of an ideal fluid of infinite depth.   The equilibrium position of the cylinder  is assumed to be attained when half of its lateral surface is immersed, see Figure \ref{fig1f}. 
\begin{figure}[htbp]
	\centering\includegraphics[width=8.6cm]{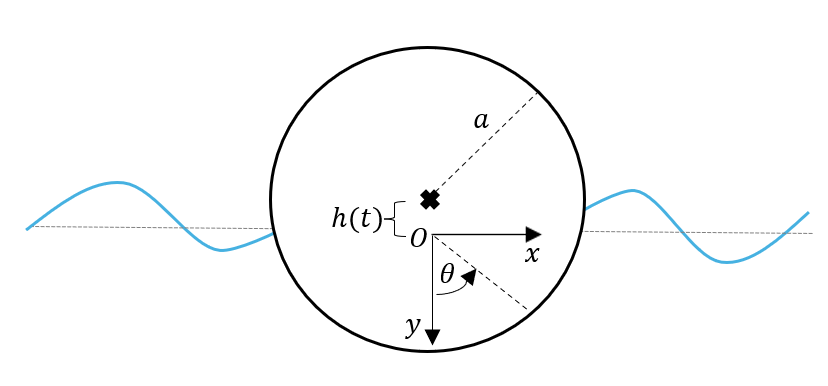}
	\caption{Floating cylinder} \label{fig1f}
\end{figure} 
Our main contribution consists in proving that the initial value problem describing the vertical oscillations of this coupled system is well-posed in standard Sobolev spaces. Due to the particular geometric configuration, some of the main technical difficulties can be solved in a quite elementary manner. More precisely, an important point of our methodology consists in deriving explicit formulas for the two potentials of the fluid velocity (one generated by the oscillations of the free surface of the fluid and a second one by those of the floating cylinder). Moreover, these explicit formulas are exploited to obtain the main ingredient of our approach: the proof of the fact that an associated Dirichlet-to-Neumann map defines an unbounded positive operator on an $L^2$ space. This fact allows us to construct a $C^0$ group of operators to describe the dynamics of the considered system.

The remaining part of this paper is organized as described below. Firstly, Section \ref{sec_statement} is devoted to the precise formulation of the considered system and to the statement of our main result. In Section \ref{sec_half}, we recall some basic properties on the classical Dirichlet-to-Neumann map for the Laplacian in the upper half-plane. In Section \ref{sec_problem_Laplacian}, we consider a problem with mixed boundary conditions for the Laplacian in a domain $\Omega$ obtained by intersecting the upper half plane with the exterior of the unit disk in $\mathbb{R}^2$. In Section \ref{sec_model_bis}, we construct and study the Dirichlet-to-Neumann map $\Lambda_\Omega$ associated with the boundary value problem introduced in Section \ref{sec_problem_Laplacian}. Finally, in Section \ref{sec_coupling}, we prove  our main result.

\section{Statement of the main result}\label{sec_statement}
In this section we give, following  \cite{ursell_1964_decay} (which, in turn, follows  \cite{john1949motion}), the precise formulation of the governing equations. Moreover, we give the precise statement of our main result. 

Within the geometric context in  Figure \ref{fig1f} above, the transverse section of the cylinder is supposed to be a disk of radius $1$. 
We assume that, at equilibrium, the center of this disk lies on the free surface of the fluid (roughly speaking, this means that the density of the solid is half of the density of the fluid).
We consider a Cartesian coordinate system $(x, y)$ with the origin fixed at the center of this disk.  
The $x$-axis is taken to be horizontal, while the $y$-axis increases with depth. The polar coordinates $(r, \theta)$ of a point in the fluid domain are defined by $x = r \sin \theta, \quad y = r \cos \theta$, with $\theta\in \left[-\frac{\pi}{2},\frac{\pi}{2}\right]$ and $r>1$. Moreover, at equilibrium, the surface of the cylinder is described by $r = 1$ and $\theta\in \left[-\frac{\pi}{2},\frac{\pi}{2}\right]$.

Let $h(t)$ denote the vertical displacement, at time $t\geqslant 0$, of the centre of the disk. The depth of the fluid is assumed to be infinite. Moreover, the amplitude of the motion is supposed to be sufficiently small so that we can use a linearized version of the governing equations. Neglecting viscosity, assuming that the flow is irrotational and that the fluid density is a constant $\rho>0$, the velocity field of the fluid is $\nabla\varphi$, where the velocity potential $\varphi(t, x, y)$ satisfies 
\begin{equation}\label{2.1}
\frac{\partial^2\varphi}{\partial x^2}(t,x,y) +
\frac{\partial^2\varphi}{\partial y^2}(t,x,y) =0 \qquad\qquad(t\geqslant 0, x^2+y^2>1, y>0).
\end{equation}
The linearized condition on the free surface, following from Bernoulli's law, is
\begin{equation}\label{2.2}
\frac{\partial^2\varphi}{\partial t^2}(t,x,0)-g \frac{\partial\varphi}{\partial y}(t,x,0)=0
\qquad\qquad(t\geqslant 0,\ |x|>1).
\end{equation}
On the cylinder the radial  components of the velocities
of the body and of the fluid are equal, {\it i.e.},

\begin{equation} \label{2.3}
\frac{\partial\varphi}{\partial r}(t,\sin\theta,\cos\theta)=\dot h(t)\cos\theta \qquad\qquad \left(t\geqslant 0,\ -\frac{\pi}{2}\leqslant\theta\leqslant \frac{\pi}{2}\right).
\end{equation}
The equation of motion for $h(t)$ is given by
\begin{equation}\label{2.4}
\frac12 \pi\rho \ddot h(t)=-2\rho g h(t)+\rho \int_{-\pi/2}^{\pi/2}\frac{\partial\varphi}{\partial t}(t,\sin\theta,\cos\theta)\cos\theta \, {\rm d}\theta+f(t),
\end{equation}
where on the right-hand side the first term stands for the hydrostatic restoring force, the second term is the resultant of the hydrodynamic pressures, whereas the third term describes a possible applied vertical force.

As usual in wave dynamics the fluid equations can be reformulated as an equation on the free boundary. More precisely, let $\Omega$ denote the water domain and $\mathcal{E}$ the free surface, i.e.,
\begin{equation}\label{Omega}
    \Omega=\left\{\left.\begin{bmatrix} x\\ y\end{bmatrix}\in \mathbb{R}^2\ \ \right|\ \ x^2+y^2>1,y>0\right\}.
\end{equation}
\begin{equation}\label{Mathcal_E}
    \mathcal{E}=\{x\in\mathbb{R} \ \ | \ \ |x|>1\}.
\end{equation}
The PDE-ODE system formed by \eqref{2.1}, \eqref{2.2}, \eqref{2.3} and \eqref{2.4}, of unknowns $\varphi$ and $h$, can be reformulated in terms of the trace of the potential on the free surface $v(t,x)=\varphi(t,x,0)$ (defined for $t\geqslant 0$ and $|x|>1$) and of $h$. Indeed, by linearity, we can write
\begin{equation}\label{dec_phi}
    \varphi(t,x,y)=(D_{\Omega}v)(t,x,y)+\dot{h}(t)\varphi_1(x,y)\qquad\left(t\geqslant 0, \begin{bmatrix}x\\y\end{bmatrix}\in\Omega\right),
\end{equation}
where
\begin{equation}
    \varphi_1(x,y)=-\frac{y}{x^2+y^2}\qquad\qquad \left(\begin{bmatrix}x\\y\end{bmatrix}\in\Omega\right),
\end{equation}
which satisfies
\begin{equation}
\left\{
\begin{aligned}
 &\frac{\partial^2\varphi_1}{\partial x^2}(x,y)+\frac{\partial^2\varphi_1}{\partial y^2}(x,y)=0 &\left(\begin{bmatrix}x\\y\end{bmatrix}\in\Omega\right)\\
 &\frac{\partial\varphi_1}{\partial r}(\sin\theta,\cos\theta)=\cos\theta &(-\pi/2\leqslant\theta\leqslant\pi/2)\\
 &\varphi_1(x,0)=0 & (x\in\mathcal{E}),
\end{aligned}
\right.
\end{equation}
and $D_\Omega$ is {\em the Dirichlet operator} on $\Omega$. This means that for each $v$ in an appropriate function space, $D_{\Omega}v$ solves
\begin{equation}\label{D_Omega_first}
\left\{
\begin{aligned}
 &\frac{\partial^2(D_{\Omega}v)}{\partial x^2}(t,x,y)+\frac{\partial^2(D_{\Omega}v)}{\partial y^2}(t,x,y)=0 &\left(\begin{bmatrix}x\\y\end{bmatrix}\in\Omega\right)\\
 &\frac{\partial(D_{\Omega}v)}{\partial r}(t,\sin\theta,\cos\theta)=0 &(-\pi/2\leqslant\theta\leqslant\pi/2)\\
 &(D_{\Omega}v)(t,x)=v(t,x) &(x\in\mathcal{E}).
\end{aligned}
\right.
\end{equation}
We refer to Section \ref{sec_problem_Laplacian} for a detailed construction of $D_\Omega$. 

From \eqref{dec_phi} it follows that
\begin{equation}
\begin{aligned}
\int_{-\pi/2}^{\pi/2}\frac{\partial\varphi}{\partial t}(t,\sin\theta,\cos\theta)\cos\theta \, {\rm d}\theta&=\int_{-\pi/2}^{\pi/2}\frac{\partial(D_{\Omega} v)}{\partial t}(t,\sin\theta,\cos\theta)\cos\theta\,{\rm d}\theta\\
&\quad-\ddot{h}(t)\int_{-\pi/2}^{\pi/2}\cos^2\theta\,{\rm d}\theta\\
&=\int_{-\pi/2}^{\pi/2}\frac{\partial(D_{\Omega}v)}{\partial t}(t,\sin\theta,\cos\theta)\cos\theta\,{\rm d}\theta-\frac{\pi}{2}\ddot{h}(t).
\end{aligned}
\end{equation}
 An important role in the remaining part of this paper will be played by the {\em Dirichlet-to-Neumann operator} on $\Omega$. This operator, denoted 
 by $\Lambda_\Omega$, will be precisely defined in Section \ref{sec_model_bis}. At this stage, we formally define the action of $\Lambda_\Omega$ on 
 every suitable $v:\mathcal{E}\to \mathbb{R}$  by
\begin{equation}
    (\Lambda_{\Omega}v)(t,x)=-\frac{\partial (D_{\Omega}v)}{\partial y}(t,x,0)\qquad(t\geqslant 0, x\in\mathcal{E}),
\end{equation}
where $D_\Omega v$ is the solution of \eqref{D_Omega_first}.

With the above notation, equations \eqref{2.2} and \eqref{2.4} can  be written as
\begin{equation}\label{reformulated_system}
\left\{
\begin{aligned}
&\frac{\partial^2 v}{\partial t^2}(t,x)+g(\Lambda_{\Omega}v)(t,x)+\frac{g}{x^2}\dot h(t)=0\qquad\qquad\qquad\quad\quad\quad(t\geqslant 0, x\in\mathcal{E})\\
&\ddot{h}(t)=-\frac{2g}{\pi}h(t)+\frac{1}{\pi}\int_{-\pi/2}^{\pi/2}\frac{\partial(D_{\Omega} v)}{\partial t}(t,\sin\theta,\cos\theta)\cos\theta\,{\rm d}\theta+\frac{f(t)}{\pi \rho}\quad (t\geqslant 0).
\end{aligned}
\right.
\end{equation}
Our study will thus focus on system \eqref{reformulated_system}.
With $\mathcal{E}$ given in \eqref{Mathcal_E}, and for $s\geqslant 0$, we denote by $W^{s,2}(\mathcal{E})$ the usual Sobolev space on $\mathcal{E}$.
These spaces are often denoted by $H^s(\mathcal{E})$, but in this paper we choose the notation $W^{s,2}(\mathcal{E})$ to avoid confusion with other notation involving the letter capital $H$ (the Hilbert transform or abstract Hilbert spaces) frequently encountered in this work.

The main result of this paper is:

\begin{theorem}\label{well_posedness}
Let $f\in L_{{\rm loc}}^1([0,\infty))$, $h_0, h_1\in \mathbb{R}$, $v_0\in W^{\frac12,2}(\mathcal{E})$ and  $v_1\in L^2(\mathcal{E})$. Then, the system \eqref{reformulated_system} with initial data $h(0)=h_0, \dot h(0)=h_1, v(0,x)=v_0(x)$ and $\frac{\partial v}{\partial t}(0,x)=v_1(x)$ has a unique solution $(v,h)$ satisfying 
\begin{gather}
 v\in C([0,\infty),W^{\frac12,2}(\mathcal{E}))\cap C^1([0,\infty), L^2(\mathcal{E})),\\
 h\in C^1([0,\infty)).
\end{gather}
Moreover, if $v_0\in W^{1,2}(\mathcal{E})$ and $v_1\in W^{\frac12,2}(\mathcal{E})$ and $f$ is of class $C^1$, then we have
\begin{gather}
 v\in C([0,\infty),W^{1,2}(\mathcal{E}))\cap C^1([0,\infty),  W^{\frac12,2}(\mathcal{E}))\cap C^2([0,\infty),L^2(\mathcal{E})),\\
 h\in C^2([0,\infty)).
\end{gather}
\end{theorem}


\section{Some background on the Dirichlet problem and on the Dirichlet-to-Neumann map in the half plane}\label{sec_half}
The precise construction of the operators $D_\Omega$ and $\Lambda_\Omega$ introduced in the previous section is based, within our approach, on known results on similar operators in the upper half-plane. This is why in this section we recall, with no claim of originality, some basic facts on the Dirichlet problem and on the Dirichlet-to-Neumann map in the half plane and we refer to the existing literature (see, for instance, Axler, Bourdon and R. Wade \cite{axler2013harmonic} for more details on this subject).

Let $\mathbb{H}$ be the upper half-plane
\begin{equation}\label{def_upper}
\mathbb{H}=\left\{\left.\begin{bmatrix} x\\ y\end{bmatrix}\in \mathbb{R}^2\ \ \right|\ \ y>0\right\}.
\end{equation}
Consider the Poisson kernel for $\mathbb{H}$ defined by
\begin{equation}\label{def_poisson}
P_\mathbb{H} (x,y,\tilde x)=\frac{y}{\pi\left[(x-\tilde x)^2+y^2\right]} \qquad\qquad\qquad(x,\ \tilde x\in \mathbb{R},\ y>0).
\end{equation}
We recall the following classical  result, see, for instance, \cite[Theorems 7.3,7.4,7.5]{axler2013harmonic}).

\begin{proposition}\label{axler_classical}
Let $\eta:\mathbb{R}\to \mathbb{R}$ be a continuous and bounded function on $\mathbb{R}$. Then the function  $\tilde\psi: \overline{\mathbb{H}}\to \mathbb{R}$, defined by
\begin{equation}\label{Poisson_extins_bis}
\tilde \psi(x,y)=(D_\mathbb{H} \eta)(x,y):=\begin{dcases}
\int_\mathbb{R} P_\mathbb{H} (x,y,\tilde x) \eta(\tilde x)\, {\rm d}\tilde x & \qquad\qquad(x\in \mathbb{R},y>0),\\
\eta(x) & \qquad\qquad(x\in \mathbb{R},y=0),
\end{dcases}
\end{equation}
is the unique  solution, in the class of functions which are continuous and bounded on $\overline{\mathbb{H}}$, of the Dirichlet problem
\begin{equation}\label{Diri_tot}
\Delta \tilde \psi(x,y)=0 \qquad\qquad(x\in \mathbb{R},\ y>0),
\end{equation}
\begin{equation}\label{frontiera_tot}
\tilde\psi(x,0)=\eta(x) \qquad\qquad\qquad(x\in \mathbb{R}).
\end{equation}
Moreover, we have
\begin{equation}\label{max_princ}
\max_{x\in\overline{\mathbb{H}}} |\tilde\psi|=\max_{x\in\overline{\mathcal{E}}} |\eta|
\end{equation}
\end{proposition}

\begin{remark}\label{rem_Four_first}
Let $\theta$ be a continuous function on $\overline{\mathbb{H}}$ such that for every $y\geqslant 0$ the function $x\mapsto \theta(x,y)$ lies in $\mathcal{S}'(\mathbb{R})$. We consider the partial Fourier transform of $\theta$ with respect to $x$ defined by
$$
\mathcal{F}\theta(\xi,y)=\widehat{\theta(\cdot,y)}(\xi) \qquad\qquad\qquad(\xi\in \mathbb{R},\ y \geqslant 0),
$$
where $\hat f$ stands for the Fourier transform of $f\in \mathcal{S}'(\mathbb{R})$.
With the above notation, formulas \eqref{Diri_tot} and \eqref{frontiera_tot} can be gathered in
\begin{equation}\label{Four_part_Dir}
\left(\mathcal{F} {D_\mathbb{H}} \eta\right)(\xi,y)=\exp(-|\xi|y) \hat\eta(\xi) \qquad\qquad\qquad(\xi\in \mathbb{R},\ y\geqslant 0).
\end{equation}
\end{remark}
\begin{proposition}\label{grad_L2}
With the notation in Proposition \ref{axler_classical}, for every $\eta\in W^{1,2}(\mathbb{R})$
we have that $\nabla(D_\mathbb{H} \eta)\in L^2(\mathbb{H})$ and
\begin{equation}\label{est_levi}
\left\|\nabla(D_\mathbb{H} \eta) \right\|_{L^2(\mathbb{H})}
\leqslant \frac{1}{\sqrt{2}} \|\eta\|_{W^{\frac12,2}(\mathbb{R})}.
\end{equation}
\end{proposition}

\begin{proof}
We first note that from \eqref{Four_part_Dir} it follows that
\begin{multline}\label{der_x}
\int_\mathbb{H} \xi^2 \left|\mathcal{F} {D_\mathbb{H}} \eta(\xi,y)\right|^2\, {\rm d}\xi\, {\rm d}y
=\int_\mathbb{R} \int_0^\infty \xi^2 \exp(-2|\xi|y) |\hat\eta(\xi)|^2 \, {\rm d}y\, {\rm d}\xi
\\=\frac12 \int_\mathbb{R} |\xi| \ |\hat\eta(\xi)|^2 \,  {\rm d}\xi \leqslant \frac12 \int_\mathbb{R} (1+|\xi|) |\hat\eta(\xi)|^2 \,  {\rm d}\xi .
\end{multline}
Moreover, from \eqref{Four_part_Dir} it also follows that
\begin{equation}\label{der_in_y}
\frac{\partial\left(\mathcal{F} {D_\mathbb{H}} \eta\right)}{\partial y}(\xi,y)=-|\xi| \exp(-|\xi|y) \hat\eta(\xi) \qquad\qquad\qquad(\xi\in \mathbb{R},\ y\geqslant 0),
\end{equation}
so that 
\begin{multline}\label{der_y}
\int_\mathbb{H} \left| \frac{\partial\left(\mathcal{F} {D_\mathbb{H}} \eta\right)}{\partial y}(\xi,y) \right|^2\, {\rm d}\xi\, {\rm d}y
=\int_\mathbb{R}\int_0^\infty |\xi|^2 \exp(-2|\xi|y) |\hat\eta(\xi)|^2\, {\rm d}y\, {\rm d}\xi
\\=\frac12 \int_\mathbb{R} |\xi|\ |\hat\eta(\xi)|^2 \,  {\rm d}\xi \leqslant \frac12 \int_\mathbb{R} (1+|\xi|) |\hat\eta(\xi)|^2 \,  {\rm d}\xi .
\end{multline}
Putting together \eqref{der_x}, \eqref{der_y} and Plancherel's identity we obtain the conclusion \eqref{est_levi}. 
\end{proof}

We next recall  some background on the Dirichlet-to-Neumann map $\Lambda_\mathbb{H}$ of the upper half plane $\mathbb{H}$. Roughly speaking, this map associates to every $\eta\in W^{1,2}(\mathbb{R})$ the Neumann trace on the real axis of 
the unique  solution, in the class of functions which are continuous and bounded on $\overline{\mathbb{H}}$, of
the Dirichlet problem \eqref{Diri_tot}, \eqref{frontiera_tot}. To give the precise definition of this map
we introduce the Hilbert transform $\mathcal{H}$ defined by
\begin{equation}\label{prima_parte_finita}
(\mathcal{H} f)(x)=\frac1\pi  {\rm P.V.} \int_\mathbb{R} \frac{f(\tilde x)}{x- \tilde x}\, {\rm d}\tilde x
\qquad\qquad\qquad(f\in L^2(\mathbb{R})),
\end{equation}
where
$$
{\rm P.V.} \int_\mathbb{R} \frac{f(\tilde x)}{x- \tilde x}\, {\rm d}\tilde x =\lim_{\varepsilon\to 0+}
\int_{|x-\tilde x|\geqslant\varepsilon} \frac{f(\tilde x)}{x- \tilde x}\, {\rm d}\tilde x
\qquad\qquad(f\in L^2(\mathbb{R}),\ x\in \mathbb{R}).
$$
We recall (see, for instance, pp. 109-110 in \cite{grafakos2024fundamentals}) that $\mathcal{H}\in \mathcal{L}(L^2(\mathbb{R}))$ and
\begin{equation}\label{Four_Hilbert}
\widehat{\mathcal{H}f}(\xi)=-i{\rm sign}\, (\xi) (\hat{f}(\xi))\qquad\qquad(f\in L^2(\mathbb{R}),\ \xi\in \mathbb{R}).
\end{equation}
We next introduce the operator 
$\Lambda_\mathbb{H}\in \mathcal{L}(W^{1,2}(\mathbb{R}),L^2(\mathbb{R}))$, called {\em Dirichlet-to-Neumann map in} $\mathbb{H}$,
defined by
\begin{equation}\label{cu_Hilbert}
(\Lambda_\mathbb{H}\eta)(x)=-(\mathcal{H}\eta')(x) \qquad\qquad\qquad(\eta\in W^{1,2}(\mathbb{R})).
\end{equation}
Let us also note that from \eqref{Four_Hilbert} and \eqref{cu_Hilbert} it follows that
\begin{equation*}
\widehat{\Lambda_\mathbb{H} \eta}(\xi)=
i{\rm sign}\, (\xi) (\widehat{\eta'}(\xi))\qquad\qquad(\eta\in W^{1,2}(\mathbb{R}),\ \xi\in \mathbb{R}).
\end{equation*}
From the above formula we obtain the alternative definition of $\Lambda_\mathbb{H}$ as
\begin{equation}\label{Dir_Four}
\widehat{\Lambda_\mathbb{H} \eta}(\xi)=|\xi| \hat \eta(\xi) \qquad\qquad(\eta\in W^{1,2}(\mathbb{R}),\ \xi\in \mathbb{R}).
\end{equation}
From \eqref{Dir_Four} and Plancherel's theorem it follows that for every $s\in \mathbb{R}$, $I+\Lambda_\mathbb{H}$ can be restricted (or extended) to a boundedly invertible operator from $W^{s,2}(\mathbb{R})$ to $W^{s-1,2}(\mathbb{R})$. This fact, together with \eqref{Dir_Four}, yields that
$\Lambda_\mathbb{H}$ is an unbounded self-adjoint positive operator on $L^2(\mathbb{R})$, with domain $W^{1,2}(\mathbb{R})$.

The result below justifies  the term {\em Dirichlet-to-Neumann map} to denote $\Lambda_\mathbb{H}$.

\begin{proposition}\label{there_exists}
For every $\eta\in W^{1,2}(\mathbb{R})$ the solution $D_\mathbb{H}\eta$ of  \eqref{Diri_tot}, \eqref{frontiera_tot}
satisfies
\begin{equation}\label{urma_bizara}
\lim_{y\to 0+} \left\|\frac{\partial(D_\mathbb{H}\eta)}{\partial y}(\cdot,y)+\Lambda_\mathbb{H} \eta\right\|_{L^2(\mathbb{R})}=0,
\end{equation}
where  $\Lambda_\mathbb{H}\in \mathcal{L}(W^{1,2}(\mathbb{R}),L^2(\mathbb{R}))$ is the operator defined in 
\eqref{cu_Hilbert} (or, alternatively, in \eqref{Dir_Four}).

Moreover, we have
\begin{equation}\label{cu_H_mare}
\int_\mathbb{H} \nabla(D_\mathbb{H} \eta_1)\cdot \overline{\nabla (D_\mathbb{H} \eta_2)} \, {\rm d}x\, {\rm d}y=\int_\mathbb{R} \left(\Lambda_\mathbb{H}\eta_1\right)\, \overline{\eta_2} \, {\rm d}x
\qquad\qquad(\eta_1, \eta_2\in W^{1,2}(\mathbb{R})).
\end{equation}

\end{proposition}

\begin{proof}
Using formula \eqref{der_in_y} in the proof of Proposition \ref{grad_L2}
\begin{multline}\label{der_in_y_bis}
\int_\mathbb{R}\left|\frac{\partial\left(\mathcal{F} {D_\mathbb{H}} \eta\right)}{\partial y}(\xi,y)+\widehat{\Lambda_\mathbb{H} \eta}(\xi)\right|^2 \, {\rm d}\xi\\
=\int_\mathbb{R}|\xi|^2 (1-\exp(-|\xi|y))^2 |\hat\eta(\xi)|^2 \, {\rm d}\xi \qquad\qquad\qquad(\xi\in \mathbb{R},\ y> 0).
\end{multline}
The conclusion \eqref{urma_bizara} follows from the above formula by combining the dominated convergence theorem with the fact that $\eta\in W^{1,2}(\mathbb{R})$
and with Plancherel's theorem.

In order to prove \eqref{cu_H_mare} we first note that the left hand side of \eqref{cu_H_mare} is well defined thanks to Proposition 
\ref{grad_L2}. Moreover,
\begin{multline*}
\frac{1}{2\pi}\int_\mathbb{H} \nabla(D_\mathbb{H} \eta_1)\cdot \overline{\nabla (D_\mathbb{H} \eta_2)}  {\rm d}y \, {\rm d}x\,= 
\int_\mathbb{H} \mathcal{F}\left(\nabla(D_\mathbb{H} \eta_1)\right)\cdot \overline{\mathcal{F}\left(\nabla (D_\mathbb{H} \eta_2)\right)} \, {\rm d}y\, {\rm d}\xi\\
=\int_\mathbb{H} \xi^2 \mathcal{F} \left({D_\mathbb{H}} \eta_1\right)(\xi,y) \overline{\mathcal{F}\left({D_\mathbb{H}} \eta_2\right)(\xi,y)}\, {\rm d}y \, {\rm d}\xi\\
+\int_\mathbb{H}  \frac{\partial\left(\mathcal{F} {D_\mathbb{H}} \eta_1\right)}{\partial y}(\xi,y) \, 
\overline{\frac{\partial\left(\mathcal{F} {D_\mathbb{H}} \eta_2\right)}{\partial y}(\xi,y)}\, {\rm d}y \, {\rm d}\xi.
\end{multline*}
Combining the above formula with \eqref{Four_part_Dir} and \eqref{der_in_y} it follows that
\begin{multline}\label{cu_H_mare_foarte}
\frac{1}{2\pi}\int_\mathbb{H} \nabla(D_\mathbb{H} \eta_1)\cdot \overline{\nabla (D_\mathbb{H} \eta_2)}  \, {\rm d}y \, {\rm d}x\,
=\int_\mathbb{H} \xi^2 \exp(-2|\xi|y) \hat\eta_1(\xi) \overline{\hat\eta_2}(\xi)\, {\rm d}y \, {\rm d}\xi\\
+\int_\mathbb{H}  \xi^2 \exp(-2|\xi|y) \hat\eta_1(\xi)\overline{\hat \eta_2}(\xi) {\rm d}y \, {\rm d}\xi=
\int_\mathbb{R}  |\xi|  \hat\eta_1(\xi)\overline{\hat \eta_2}(\xi) {\rm d}y \, {\rm d}\xi,
\end{multline}
which clearly yields the conclusion \eqref{cu_H_mare}.
\end{proof}

\begin{remark}\label{rem_clasic}
It is not difficult to check that if $\eta\in W^{2,2}(\mathbb{R})$ then $\nabla(D_\mathbb{H}\eta)$ is continuous on $\overline{\mathbb{H}}$
and
\begin{equation}\label{foarte_clasic}
(\Lambda_\mathbb{H} \eta)(x)=-\frac{\partial (D_\mathbb{H}\eta)}{\partial y}(x,0) \qquad\qquad(\eta\in W^{2,2}(\mathbb{R}),\ x\in \mathbb{R}).
\end{equation}
Indeed, from \eqref{Four_part_Dir} and \eqref{der_in_y} we have
\begin{equation}
\begin{aligned}
 \frac{\partial (\mathcal{F}D_{\mathbb{H}}\eta)}{\partial \xi}(\xi,y)&=-y\,{\rm sgn}(\xi){\rm exp}(-|\xi|y)\hat\eta(\xi)+{\rm exp}(-|\xi|y)\widehat{\eta'}(\xi)\\
 &=-y\,{\rm sgn}(\xi){\rm exp}(-|\xi|y)\hat\eta(\xi)+i\xi{\rm exp}(-|\xi|y)\hat\eta(\xi)\\
 &=(-y\,{\rm sgn}(\xi)+i\xi){\rm exp}(-|\xi|y)\hat\eta(\xi)\qquad\qquad(y>0,\xi\in\mathbb{R})
\end{aligned}
\end{equation}
and
\begin{equation}
\frac{\partial(\mathcal{F}D_{\mathbb{H}}\eta)}{\partial y}(\xi,y)=-|\xi|{\rm exp}(-|\xi|y)\hat\eta(\xi) \qquad\qquad(y>0,\xi\in\mathbb{R}). 
\end{equation}
Since $\eta\in W^{2,2}(\mathbb R)$, $\eta'\in W^{1,2}(\mathbb R)$, so that $\eta'$ is continuous, and thus, so is $\xi\hat\eta(\xi)$. It follows from the above formulas that $\nabla (\mathcal{F}(D_{\mathbb H}\eta))$ (and thus $\nabla(D_{\mathbb H}\eta)$) is continuous on $\overline{\mathbb H}$.
Moreover, using \eqref{der_in_y} again we have
\begin{equation}
\lim_{y\to 0^+}\frac{\partial(\mathcal{F}D_{\mathbb{H}}\eta)}{\partial y}(\xi,y)=-\lim_{y\to 0^+}|\xi|{\rm exp}(-|\xi|y)\hat\eta(\xi)=-|\xi|\hat{\eta}(\xi)=-\widehat{\Lambda_{\mathbb{H}}\eta}(\xi),
\end{equation}
where we used \eqref{Dir_Four} for the last equality.
Thus the operator $\Lambda_\mathbb{H}$ can also be seen as the extension to $W^{1,2}(\mathbb{R})$ of the operator on $W^{2,2}(\mathbb{R})$ defined by the right-hand side of \eqref{foarte_clasic}. 
\end{remark}

\section{A boundary value problem for the Laplacian}
\label{sec_problem_Laplacian}
In this section we construct in a rigorous manner the operator $D_\Omega$ defined in Section \ref{sec_statement} and we explicitly compute a Poisson type kernel describing this operator.

Let $\Omega$ and $\mathcal{E}$ be given by \eqref{Omega} and \eqref{Mathcal_E} respectively. We are interested in the mixed problem which consists of determining $\psi:\overline\Omega\to \mathbb{R}$ such that
\begin{equation}\label{equation_Dw_Laplace}
\left\{
\begin{aligned}
&\frac{\partial^2\psi}{\partial x^2}(x,y) + \frac{\partial^2\psi}{\partial y^2}(x,y) =0 && \qquad(x^2+y^2>1, y>0),\\
&\frac{\partial\psi}{\partial r}(x,y)=0 && \qquad (x\in \mathcal{E},\ y>0,\ x^2+y^2=1),\\
&\psi(x,0)=v(x) &&\qquad (|x|>1),
\end{aligned}
\right.
\end{equation}
where the radial derivative of $\psi$ is defined by
\begin{equation}
\frac{\partial\psi}{\partial r}(x,y)=\frac{x}{\sqrt{x^2+y^2}}\frac{\partial\psi}{\partial x}(x,y)
+ \frac{y}{\sqrt{x^2+y^2}}\frac{\partial\psi}{\partial y}(x,y)\qquad\qquad (x\in \mathcal{E},\ y>0).
\end{equation}

\begin{proposition}\label{21}
For  $v\in W^{1,2}(\mathcal{E})$, the problem \eqref{equation_Dw_Laplace} has a unique solution $\psi=D_\Omega v$ in the class of functions which are continuous and bounded on $\overline\Omega$. Moreover,
  for every  $\begin{bmatrix} x\\ y\end{bmatrix}\in \Omega$  we have
\begin{equation}\label{explicit_formula_Dv}
    (D_\Omega v)(x,y)=
    \frac{1}{\pi}\int_{\mathcal{E}}\left[\frac{y}{( x-\tilde x)^2+y^2}+\frac{y}{(x\tilde x-1)^2+y^2\tilde x^2}\right]v(\tilde x)\,{\rm d}\tilde x .
\end{equation}
Finally, $D_\Omega$ is a linear bounded operator from $W^{1,2}(\mathcal{E})$ to $C_b(\overline\Omega)$, where $C_b(\overline{\Omega})$
stands for the space of continuous bounded functions on $\overline{\Omega}$, endowed with the $L^\infty$ norm.
\end{proposition}

\begin{proof}
Let $\mathbb{H}$ be the upper half-plane introduced in \eqref{def_upper}
and  let $P_\mathbb{H}$ be the Poisson kernel for $\mathbb{H}$ defined in \eqref{def_poisson}.
We next extend $v$ to $\eta(v):\mathbb{R}\to \mathbb{R}$, with 
\begin{equation}\label{numar_eta_prim}
(\eta(v))(x)=\begin{cases}
    v(x) & \qquad\qquad(x\in \mathcal{E})\\
    v\left(\frac1x\right) & \qquad\qquad(x\in (-1,1)\setminus \{0\}),\\
    0 & \qquad\qquad(x=0).
\end{cases}
\end{equation}
Since $v\in W^{1,2}(\mathcal{E})$, a classical Sobolev embedding implies that $\eta(v)$ is continuous and bounded on $\mathbb{R}$. For notational simplicity, we write $\eta$ instead of $\eta(v)$ in the remaining part of this proof.

We claim that the restriction $\psi$ of $\tilde\psi=D_\mathbb{H}\eta$, where $\tilde\psi=D_\mathbb{H}\eta$ has been introduced in Proposition \ref{axler_classical}, to $\overline\Omega$ is a classical solution of \eqref{equation_Dw_Laplace}. To prove this claim, we first note that from
the properties established in Proposition \ref{axler_classical} for $\tilde\psi$ it
follows that $\psi$ is harmonic on $\Omega$, as well as  continuous and bounded on $\overline \Omega$.
We thus just need to check that $\psi$ satisfies the Neumann boundary condition in \eqref{equation_Dw_Laplace}.
To this aim we note that from \eqref{def_poisson} it follows that
for $x,\ \tilde x\in \mathbb{R}$ and $y>0$ we have

\begin{equation*}
\begin{aligned}
   \pi \left[ x\frac{\partial P_\mathbb{H}}{\partial x}(x,y,\tilde x)+y\frac{\partial P_\mathbb{H}}{\partial y}(x,y,\tilde x)\right]&=x\frac{2(\tilde{x} - x) y}{((\tilde{x} - x)^2 + y^2)^2}+y\frac{(\tilde{x} - x)^2 - y^2}{((\tilde{x} - x)^2 + y^2)^2}\\
    &=\frac{2\tilde x x y-2x^2y+\tilde x^2y-2\tilde x x y+x^2y-y^3}{(\tilde x^2-2x\tilde x+x^2+y^2)^2}\\
    &=\frac{\tilde x^2y-y(x^2+y^2)}{(\tilde x^2-2x\tilde x+x^2+y^2)^2}.
\end{aligned}
\end{equation*}
From the above formula it follows that
\begin{equation*}
\frac{\partial P_\mathbb{H}}{\partial r}(x,y)=\frac{\tilde x^2 y-y}{\pi(\tilde x^2-2x \tilde x+1)^2} \qquad\qquad(y>0,\ x^2+y^2=1).
\end{equation*}
Combining the above formula with \eqref{Poisson_extins_bis} and the dominated convergence theorem it follows that for $x\in\mathbb{R},\ y>0,\ x^2+y^2=1$, we have
\begin{multline*}
\frac{\partial \psi}{\partial r}(x,y)=\frac{\partial\tilde\psi}{\partial r}(x,y)=\int_\mathbb{R} \frac{\partial P_\mathbb{H}}{\partial r}(x,y,\tilde x) \eta(\tilde x)\, {\rm d}\tilde x
=\frac1\pi \int_\mathbb{R} \frac{\tilde x^2 y-y}{(\tilde x^2-2x \tilde x+1)^2} \eta(\tilde x)\, {\rm d}\tilde x\\
=\frac1\pi \int_\mathcal{E} \frac{\tilde x^2 y-y}{(\tilde x^2-2x \tilde x+1)^2} \eta(\tilde x)\, {\rm d}\tilde x
+\frac1\pi \int_{-1}^1 \frac{\tilde x^2 y-y}{(\tilde x^2-2x \tilde x+1)^2} \eta(\tilde x)\, {\rm d}\tilde x.
\end{multline*}
Using a change of variable and the fact that $\eta(x)=\eta\left(\frac1x\right)$ for $x\in (-1,1)\setminus \{0\}$ in the last integral in the formula above, we obtain that
\begin{multline*}
\frac{\partial \psi}{\partial r}(x,y)
=\frac1\pi \int_\mathcal{E} \frac{\tilde x^2 y-y}{(\tilde x^2-2x \tilde x+1)^2} \eta(\tilde x)\, {\rm d}\tilde x\\
+\frac1\pi \int_\mathcal{E} \frac{y-\tilde x^2 y}{(\tilde x^2-2x \tilde x+1)^2} \eta(\tilde x)\, {\rm d}\tilde x=0 \qquad\qquad(x^2+y^2=1,\ y>0).
\end{multline*}
We have thus shown that the function $\psi$ defined by \eqref{explicit_formula_Dv} defines a classical solution of \eqref{equation_Dw_Laplace}.

To prove the uniqueness of the solution  constructed above, 
we use an argument inspired by the Schwarz reflection principle for harmonic functions.
More precisely, let $\Phi$ be a harmonic function on $\Omega$, which is continuous and bounded on $\overline\Omega$ and with
$$
\Phi(x,0)=v(x) \qquad\qquad\qquad(x\in \mathcal{E}),
$$
$$
\frac{\partial \Phi}{\partial r}(x,y)=0 \qquad\qquad\qquad(x^2+y^2=1,\ y>0),
$$
with the last equality understood in a weak sense.
We extend $\Phi$ to a function $\tilde\Phi$ defined on $\mathbb{H}$ by setting
  \begin{equation*}
\tilde\Phi(x,y)=\begin{cases}
\Phi(x,y) & \qquad\qquad(x\in \mathbb{R},y \geqslant0,x^2+y^2\geqslant 1)\\
\Phi\left(\frac{x}{x^2+y^2},\frac{y}{x^2+y^2}\right) & \qquad\qquad(x\in \mathbb{R},y\geqslant 0,x^2+y^2<1).
\end{cases}
\end{equation*}
Then $\tilde\Phi$ is continuous and bounded on $\overline{\mathbb{H}}$. Moreover, since $\frac{\partial\tilde\Phi}{\partial r}$ 
is also continuous on $\mathbb{H}$, it follows that $\tilde\Phi$ is harmonic on $\mathbb{H}$. Moreover, we clearly have $\Phi(x,0)=\eta(x)$
for $x\in\mathbb{R}$,
where $\eta$ is the function defined in \eqref{numar_eta_prim}. Thus, according to \cite[Theorem 7.5]{axler2013harmonic}
$\tilde\Phi$ coincides with the Poisson integral of $\eta$, i.e., with the function $\tilde\psi$ defined in \eqref{Poisson_extins_bis}.
From this it follows that $\Phi=\psi$ on $\overline\Omega$, which ends the uniqueness proof.

We still have to prove \eqref{explicit_formula_Dv}. To this aim, we first note that from  \eqref{Poisson_extins_bis}
 it follows that for every
$\begin{bmatrix} x\\ y\end{bmatrix}\in \mathbb{H}$ we have
\begin{equation}\label{Poisson_extins_bis_bis}
\tilde\psi(x,y)= \int_\mathcal{E} P_\mathbb{H} (x,y,\tilde x) \eta(\tilde x)\, {\rm d}\tilde x +\int_{-1}^1 P_\mathbb{H} (x,y,\tilde x) \eta(\tilde x)\, {\rm d}\tilde x  .
\end{equation}
We next remark that from the definition \eqref{def_poisson} of the Poisson kernel and the fact that $\eta(x)=\eta \left(\frac1x\right)$ on $\mathbb{R}^*$
it follows that
\begin{equation*}
\begin{aligned}
\int_{-1}^1 P_\mathbb{H} (x,y,\tilde x) \eta(\tilde x)\, {\rm d}\tilde x&=\frac{y}{\pi}\int_{-1}^1\frac{\eta(\tilde x)}{(x-\tilde x)^2+y^2}\, {\rm d}\tilde x \\
&=\frac{y}{\pi}\int_{-1}^1\frac{\eta(\tilde x)}{{\tilde x}^2\left[\left(\frac{x}{\tilde x}-1\right)^2+\frac{y^2}{\tilde x^2}\right]}\, {\rm d}\tilde x\\
&=\frac{y}{\pi}\int_\mathcal{E}\frac{\eta(\tilde x)}{\left(x \tilde x-1\right)^2+ y^2 \tilde x^2}\, {\rm d}\tilde x   \qquad\qquad\qquad(x\in \mathbb{R},\ y>0).
\end{aligned}
\end{equation*}
Combining the above formula with \eqref{Poisson_extins_bis_bis}, \eqref{def_poisson} and the fact that $\psi$ is the restriction to $\overline\Omega$ of
$\tilde\psi$ yields the conclusion \eqref{explicit_formula_Dv}.

Finally, the fact that $D_\Omega$ is a linear bounded operator from $W^{1,2}(\mathcal{E})$ to $C_b(\overline\Omega)$ follows from
\eqref{max_princ}, \eqref{numar_eta_prim}, \eqref{Poisson_extins_bis_bis} and the continuity of the embedding of $W^{1,2}(\mathcal{E})$ into
$C_b(\overline{\mathcal{E}})$.
\end{proof}

\section{A Dirichlet-to-Neumann map on \texorpdfstring{$\Omega$}{Omega}} \label{sec_model_bis}
In this section we construct and study a Dirichlet-to-Neumann map for the domain $\Omega$
defined in \eqref{Omega}. More precisely, we study the map $\Lambda_\Omega$, associating to a function $v$
defined on $\mathcal{E}$, with $\mathcal{E}$ defined in \eqref{Omega}, the trace on $\mathcal{E}$
of the normal derivative of the solution $\psi$ of \eqref{equation_Dw_Laplace}.
Rather than defining this map by duality, as is usually done in the literature (see, for instance, Lannes \cite{lannes2024posednessfjohnsfloating}
and references therein), we use below the Dirichlet-to-Neumann map $\Lambda_\mathbb{H}$ of the half-plane, introduced in Section \ref{sec_half}, and a symmetry argument.
Our aim being to define $\Lambda_\Omega$ on $W^{1,2}(\mathcal{E})$, we first define this operator 
on the subspace $V_0(\mathcal{E})$ of $W^{1,2}(\mathcal{E})$ formed by the functions whose support is a compact subset of $\mathbb{R}$.
Note that for $v\in W^{1,2}(\mathcal{E})$ the function $\eta(v)$  defined in \eqref{numar_eta_prim} lies in $W^{1,2}(\mathbb{R})$.

\begin{definition}\label{from_definition}
The Dirichlet-to-Neumann
map associated to the boundary value problem \eqref{equation_Dw_Laplace},
denoted by $\Lambda_\Omega$, is defined on $V_0(\mathcal{E})$ by
\begin{equation}\label{l_restrans}
(\Lambda_\Omega v)(x)=\left[\Lambda_\mathbb{H} (\eta(v))\right](x) \qquad\qquad\qquad(v\in V_0(\mathcal{E}), x\in \mathcal{E}),
\end{equation}
where $\eta(v)$ has been defined in \eqref{numar_eta_prim}.
\end{definition}

\begin{remark}\label{to_understand}
If $v\in \mathcal{D}(\mathcal{E})$ then $\eta(v)$ clearly lies in $\mathcal{D}(\mathbb{R})$
and, according to \eqref{foarte_clasic}, $\left[(\Lambda_\mathbb{H}\eta(v))\right](x)=-\frac{\partial (D_\mathbb{H}\eta(v))}{\partial y}(x,0)$
for every $x\in \mathbb{R}$, in a classical sense. Thus, using the fact, already used in the proof of Proposition \ref{21}, that
$D_\Omega v=D_\mathbb{H}(\eta(v))$ on $\Omega$ it follows that for $v\in \mathcal{D}(\mathcal{E})$, we have that
$\left[(\Lambda_\Omega v)\right](x)=-\frac{\partial (D_\Omega v)}{\partial y}(x,0)$ for every $x\in \mathcal{E}$, in a classical sense.
\end{remark}
The result below plays an important role in the remaining part of this paper.

\begin{proposition}\label{formula_hilbert}
Let $\Omega$ be defined in  \eqref{Omega} and let $\Lambda_\Omega$ be the Dirichlet-to-Neumann map
introduced in Definition \ref{from_definition}.
Then $\Lambda_\Omega$ extends to an operator, still denoted by $\Lambda_\Omega$, in $\mathcal{L}(W^{1,2}(\mathcal{E}),L^2(\mathcal{E}))$. Moreover, we have
\begin{multline}\label{Mare_Formula}
\left(\Lambda_\Omega v\right)(x)=-(\mathcal{H}(g(v))(x) - \frac{1}{x^2} (\mathcal{H}(g(v))\left(\frac1x\right)\\
-\frac{v(1)-v(-1)}{\pi x} \qquad\qquad(v\in W^{1,2}(\mathcal{E}),\ x\in \mathcal{E}),
\end{multline}
where  $g(v)\in L^2(\mathbb{R})$ is the extension of $v'$ by zero outside $\mathcal{E}$ and $\mathcal{H}$ is the Hilbert transform defined in \eqref{prima_parte_finita}.
\end{proposition}

\begin{proof}
Let $v\in W^{1,2}(\mathcal{E})$ be such that there exists $a>0$ with $v(x)=0$ for  $|x|\geqslant a$. Then the function $\eta(v)$ defined by
\eqref{numar_eta_prim} lies in $W^{1,2}(\mathbb{R})$. For the sake of simplicity, this function will simply be denoted by $\eta$ in the remaining part of this proof.

As seen in the proof of Proposition \ref{21}
the solution $\psi$ of \eqref{equation_Dw_Laplace} is the restriction to $\overline\Omega$ of the
solution $\tilde\psi$ of   \eqref{Diri_tot}-\eqref{frontiera_tot}. We thus have
\begin{equation}\label{simetria}
\left(\Lambda_\Omega v\right)(x)=\left(\Lambda_\mathbb{H} \eta\right)(x) \qquad\qquad\qquad(x\in \mathcal{E}\ {\rm a.e.})\, .
\end{equation}
Using the fact that $\Lambda_\mathbb{H}\in \mathcal{L}(W^{1,2}(\mathbb{R}),L^2(\mathbb{R}))$ it follows that $\Lambda_\Omega v\in L^2(\mathcal{E})$.
Moreover, using \eqref{cu_Hilbert} it follows that
\begin{multline}\label{TREI_parte_finita}
\left(\Lambda_\Omega v\right)(x)= - \frac1\pi  \lim_{\varepsilon\to 0+}
\int_{|\tilde x-x|\geqslant\varepsilon} \frac{\eta'(\tilde x)}{x-\tilde x}\, {\rm d}\tilde x\\
=-\frac1\pi  \lim_{\varepsilon\to 0+} \left(
\int_{\mathcal{E}_\varepsilon(x)} \frac{\eta'(\tilde x)}{x-\tilde x}\, {\rm d}\tilde x +
\int_{\mathcal{I}_\varepsilon(x)} \frac{\eta'(\tilde x)}{x-\tilde x}\, {\rm d}\tilde x\right)
\qquad\qquad( x\in \mathcal{E}),
\end{multline}
where
\begin{equation}\label{defeeps}
 \mathcal{E}_{\varepsilon}(x)=\{\tilde x\in \mathcal{E}\ \ |\ \ |x-\tilde x|>\varepsilon\} \qquad\qquad\qquad(x\in \mathcal{E}).
 \end{equation}
\begin{equation}\label{defieps}
 \mathcal{I}_{\varepsilon}(x)=\{\tilde x\in (-1,1)\ \ |\ \ |x-\tilde x|>\varepsilon\} \qquad\qquad\qquad(x\in \mathcal{E}).
 \end{equation}
 It is easily seen that
\begin{equation}\label{defieps_noua}
 \mathcal{I}_{\varepsilon}(x)=(-1,1) \qquad\qquad\qquad(x\in \mathcal{E}, \varepsilon <|x|-1),
 \end{equation}
 so that \eqref{TREI_parte_finita} can be rewritten
\begin{equation}\label{TREI_parte_finita_bus}
\left(\Lambda_\Omega v\right)(x)= - \frac1\pi  \lim_{\varepsilon\to 0+} \left(
\int_{\mathcal{E}_\varepsilon(x)} \frac{\eta'(\tilde x)}{x-\tilde x}\, {\rm d}\tilde x +
\int_{-1}^1 \frac{\eta'(\tilde x)}{x-\tilde x}\, {\rm d}\tilde x\right)
\qquad\qquad( x\in \mathcal{E}).
\end{equation}
 We next remark that
 $$
 \lim_{\varepsilon\to 0+}
\int_{\mathcal{E}_\varepsilon(x)} \frac{\eta'(\tilde x)}{x-\tilde x}\, {\rm d}\tilde x
=\lim_{\varepsilon\to 0+}
\int_{|\tilde x-x|>\varepsilon} \frac{g(\tilde x)}{x-\tilde x}\, {\rm d}\tilde x
=\pi(\mathcal{H}g)(x) \qquad\qquad(x\in \mathcal{E}).
$$
Denote
\begin{equation}\label{defh1}
h_1(x)=\lim_{\varepsilon\to 0+}
\int_{\mathcal{E}_\varepsilon(x)} \frac{\eta'(\tilde x)}{x-\tilde x}\, {\rm d}\tilde x
\qquad\qquad(x\in \mathcal{E}).
\end{equation}
Using the fact that the Hilbert transform, as defined in \eqref{prima_parte_finita}, is an isometry of $L^2(\mathbb{R})$, we obtain that
\begin{multline}\label{est_unul_termen}
    \|h_1\|_{L^2(\mathcal{E})}=\pi\|\mathcal{H}g(v)\|_{L^2(\mathcal{E})}
    \leqslant \pi\|\mathcal{H}g(v)\|_{L^2(\mathbb{R})}\\
    =\pi \|g(v)\|_{L^2(\mathbb{R})}
    \leqslant \pi \|v\|_{W^{1,2}(\mathcal{E})}.
\end{multline}
Let
\begin{equation}\label{defh2}
h_2(x)=\int_{-1}^1 \frac{\eta'(\tilde x)}{x-\tilde x}\, {\rm d}\tilde x \qquad\qquad\qquad(x\in \mathcal{E}).
\end{equation}
Then
\begin{multline}\label{asta_e_h2}
h_2(x)=\frac{\eta(1)-\eta(-1)}{x} + \frac1x \int_{-1}^1 \frac{\tilde x \eta'(\tilde x)}{x-\tilde x}\, {\rm d}\tilde x\\
=\frac{v(1)-v(-1)}{x} - \frac{1}{x^2} \int_{-1}^1 \frac{ v'\left(\frac{1}{\tilde x}\right)}{\frac{1}{\tilde x}
-\frac{1}{x}} \frac{1}{{\tilde x}^2} \, {\rm d}\tilde x\\
=\frac{v(1)-v(-1)}{x} - \frac{1}{x^2} \int_\mathcal{E} \frac{ v'(s)}{s
-\frac{1}{x}}  \, {\rm d} s\\
=\frac{v(1)-v(-1)}{x} + \frac{1}{x^2} (\mathcal{H}g)\left(\frac1x\right) \qquad\qquad(x\in \mathcal{E}).
\end{multline}
From the last formula it clearly follows that
\begin{equation}\label{est_doi_termen}
\|h_2\|_{L^2(\mathcal{E})} \leqslant c\|v\|_{W^{1,2}(\mathcal{E})},
\end{equation}
where $c$ is a universal constant.

On the other hand, from \eqref{TREI_parte_finita_bus} we have
\begin{equation}\label{e_o_suma}
\left(\Lambda_\Omega v\right)(x)=-\frac{1}{\pi}(h_1(x) + h_2(x)) \qquad\qquad(x\in \mathcal{E}).
\end{equation}
Combining the last formula with \eqref{est_unul_termen} and \eqref{est_doi_termen} it follows that there exists a universal
constant $c>0$ such that the inequality
$$
\|\Lambda_\Omega v\|_{L^2(\mathcal{E})} \leqslant c\|v\|_{W^{1,2}(\mathcal{E})}
$$
holds for every $v\in W^{1,2}(\mathcal{E})$ which vanishes on $(-\infty,-a]\cup [a,+\infty)$ for some $a>0$.
By a density argument, it follows that the last inequality holds for every $v\in W^{1,2}(\mathcal{E})$, which implies the first conclusion of the proposition.

Finally, \eqref{Mare_Formula} follows from \eqref{defh1},\eqref{asta_e_h2} and \eqref{e_o_suma}.
\end{proof}

The result below shows, in particular, that the solution $(D_\Omega v)$ of \eqref{equation_Dw_Laplace} given by \eqref{explicit_formula_Dv}
is also the solution of \eqref{equation_Dw_Laplace} in the variational sense introduced in \cite{lannes2024posednessfjohnsfloating}. Moreover, this result provides an important identity involving $\Lambda_\Omega $.

To precisely state this result we introduce the Beppo-Levi space 
$$
\dot{W}^{1,2}(\Omega) := \{\varphi \in  L_{\rm loc}^1(\Omega), \nabla \varphi \in L^2(\Omega)\}.
$$

\begin{proposition}\label{prop_var}
With the notation in Proposition \ref{formula_hilbert}, for every $v\in W^{1,2}(\mathcal{E})$ the operator $D_\Omega$ defined in \eqref{explicit_formula_Dv} satisfies
\begin{equation}\label{cu_beppo}
D_\Omega v\in \dot{W}^{1,2}(\Omega) \qquad\qquad(v\in W^{1,2}(\mathcal{E}))
\end{equation}
with
\begin{equation}\label{gradient_bound}
    \|\nabla(D_{\Omega}v)\|_{L^2(\Omega)}\leqslant C \|v\|_{W^{1,2}(\mathcal{E})},
\end{equation}
for a constant $C>0$. Moreover, $\Lambda_\Omega$ (defined in \eqref{Mare_Formula}) satisfies
\begin{equation}\label{cu_Omega}
\int_\Omega \nabla(D_\Omega v)\cdot \overline{\nabla (D_\Omega u)} \, {\rm d}x\, {\rm d}y=\int_\mathcal{E} \left(\Lambda_\Omega v\right)\, \overline u \, {\rm d}x
\qquad\qquad\qquad(v, u\in W^{1,2}(\mathcal{E})).
\end{equation}
\end{proposition}

\begin{proof}
First, let $\mathcal{D}(\overline{\mathcal{E}})$ be the space of restrictions to $\overline{\mathcal{E}}$ of functions in $\mathcal{D}(\mathbb{R})$. This space is obviously contained in $W^{1,2}(\mathcal{E})$. Let
$u,v\in \mathcal{D}(\overline{\mathcal{E}})$ and let $\eta(v),\eta(u)$ be their extensions defined by \eqref{numar_eta_prim}. We denote by $\mathbb{D}^+$ the upper unit half-disk, i.e., 
\begin{equation}
    \mathbb{D}^+ = \left\{ \left.\begin{bmatrix} x \\ y \end{bmatrix} \in \mathbb{R}^2 \ \ \right| \ \ x^2 + y^2 < 1, y > 0 \right\}.
\end{equation}
Using \eqref{cu_H_mare} we get
\begin{equation}
\begin{aligned}
  \int_\Omega\nabla(D_\Omega v)\cdot \overline{\nabla (D_\Omega u)}{\rm d}x\,{\rm d}y&= \int_\mathbb{H}\nabla(D_\mathbb{H}\eta(v))\cdot\overline{\nabla (D_\mathbb{H} \eta(u))}{\rm d}x\,{\rm d}y\\
  &\quad-\int_{\mathbb{D}^+}\nabla(D_\mathbb{H}\eta(v))\cdot\overline{\nabla (D_\mathbb{H} \eta(u))}{\rm d}x\,{\rm d}y\\
  &=\int_{\mathbb{R}}\Lambda_{\mathbb{H}}\eta(v)\,\overline{\eta(u)}{\rm d}x-\int_{\partial\mathbb{D}^+}\frac{\partial (D_\mathbb{H}\eta(v))}{\partial\nu}\overline{\eta(u)}{\rm d}x\\
  &=\int_{\mathbb{R}}\Lambda_{\mathbb{H}}\eta(v)\,\overline{\eta(u)}{\rm d}x+\int_{-1}^1\frac{\partial (D_\mathbb{H}\eta(v))}{\partial y}(x,0)\overline{\eta(u)}{\rm d}x\\
  & \quad -\int_{-\pi/2}^{\pi/2}\underbrace{\frac{\partial (D_{\Omega}v)}{\partial r}(\sin\theta,\cos\theta)}_\text{$0$}\cos\theta{\rm\ d}\theta\\
  &=\int_{\mathbb{R}}\Lambda_{\mathbb{H}}\eta(v)\,\overline{\eta(u)}{\rm d}x-\int_{-1}^1\Lambda_{\mathbb{H}}\eta(v)\,\overline{\eta(u)}{\rm d}x\\
 &=\int_{\mathcal{E}}\Lambda_{\mathbb{H}}\eta(v)\,\overline{\eta(u)}{\rm d}x\\
 &= \int_{\mathcal{E}}\Lambda_\Omega v\,\overline{u}{\rm d}x,
\end{aligned}
\end{equation}
so we get \eqref{cu_Omega} for $u,v\in \mathcal{D}(\overline{\mathcal{E}})$. Now let $u,v\in W^{1,2}(\mathcal{E})$. Since $\mathcal{D}(\overline{\mathcal{E}})$ is dense in $W^{1,2}(\mathcal{E})$, there exist sequences $(v_n),(u_n)\subset \mathcal{D}(\overline{\mathcal{E}})$ such that $v_n\to v$ and $u_n\to u$ in $W^{1,2}(\mathcal{E})$. For each $n\in\mathbb{N}$ we have
\begin{equation}\label{equality_n}
    \int_\Omega\nabla(D_\Omega v_n)\cdot \overline{\nabla (D_\Omega u_n)}{\rm d}x\,{\rm d}y=\int_{\mathcal{E}}\Lambda_\Omega v_n\,\overline{u_n}{\rm d}x.
\end{equation}
By Proposition \ref{formula_hilbert}, $\Lambda_\Omega$ is a bounded operator from $W^{1,2}(\mathcal{E})$ to $L^2(\mathcal{E})$, so that, taking $u_n = v_n$, we obtain that for every $m,\ n\in \mathbb{N}$ we have
\begin{equation}
    \|\nabla(D_\Omega(v_n - v_m))\|_{L^2(\Omega)}^2 = \langle \Lambda_\Omega(v_n - v_m), v_n - v_m \rangle_{L^2(\mathcal{E})} \leqslant C \|v_n - v_m\|_{W^{1,2}(\mathcal{E})}^2.
\end{equation}
This implies that $(\nabla(D_\Omega v_n))$ is a Cauchy sequence in $L^2(\Omega)$, so that it converges to a certain ${\bf \Psi}\in L^2(\Omega)$. On the other hand, by Proposition \ref{21}, we have $D_\Omega v_n\to D_\Omega v$ in $C_b(\overline\Omega)$, so that $\nabla(D_\Omega v_n)\to \nabla (D_\Omega v)$ in $\mathcal{D}'(\Omega)$.
It follows that $\nabla (D_\Omega v)={\bf \Psi}$, so that $\nabla (D_\Omega v) \in L^2(\Omega)$. We have thus shown that $D_\Omega v\in \dot{W}^{1,2}(\Omega)$ and that \ $D_\Omega v_n\to D_\Omega v\in \dot{W}^{1,2}(\Omega)$. We can thus pass to the limit in \eqref{equality_n} and use again Proposition \ref{formula_hilbert} to obtain \eqref{gradient_bound} and \eqref{cu_Omega}.
\end{proof}
A direct consequence of Proposition \ref{prop_var} is

\begin{corollary}\label{simetric}
The operator $\Lambda_\Omega:W^{1,2}(\mathcal{E}) \to L^2(\mathcal{E})$ is symmetric on $L^2(\mathcal{E})$.
\end{corollary}

The main result of this section is:

\begin{theorem}\label{formula_hadamard}
With the assumptions and notation in Proposition \ref{formula_hilbert} 
the operator $\Lambda_\Omega$, with domain $\mathcal{D}(\Lambda_\Omega)=W^{1,2}(\mathcal{E})$,
is an (unbounded) self-adjoint operator on $H=L^2(\mathcal{E})$.
\end{theorem}

In order to prove Theorem \ref{formula_hadamard} we introduce, following an idea 
already exploited in Friedman and Shinbrot \cite{friedman1967initial},
the boundary value problem
\begin{equation}\label{ec_phi}
\left\{
\begin{aligned}
&\Delta \phi(x,y)=0 &&\left(\begin{bmatrix} x\\ y\end{bmatrix} \in \Omega\right)\\
&-\frac{\partial\phi}{\partial y}(x,0)+ \phi(x,0)=f(x) &&(x\in \mathcal{E})\\
&\frac{\partial\phi}{\partial r}(x,y)=0 &&(x\in \mathcal{E},\ y>0,\ x^2+y^2=1),
\end{aligned}
\right.
\end{equation}
where $f\in L^2(\mathcal{E})$  is given. We define weak solutions of \eqref{ec_phi}.
as follows:

\begin{definition}\label{def_weak_aux}
Let $V(\Omega)$ be the closure of $\mathcal{D}(\overline\Omega)$ with respect to the norm
$$
\|\Psi\|_{V(\Omega)}^2=\int_\Omega|\nabla\Psi|^2\, {\rm d}x\, {\rm d}y+
\int_\mathcal{E} |\Psi(x,0)|^2\, {\rm d}x \qquad\qquad(\Psi\in \mathcal{D}(\overline\Omega)). 
$$
A function $\phi\in V(\Omega)$ is called solution of \eqref{ec_phi} if
\begin{equation}\label{ne_egale}
\int_\Omega \nabla \phi\cdot \nabla \Psi \, {\rm d}x\, {\rm d}y+
\int_\mathcal{E} \phi(x,0)\Psi(x,0)\,{\rm d}x= \int_\mathcal{E} f(x)\Psi(x,0)\, {\rm d}x
\qquad (\Psi\in V(\Omega)).
\end{equation}
\end{definition}

\begin{proposition}\label{prop_robin}
For every $f\in L^2(\mathcal{E})$ the system \eqref{ec_phi} admits a unique solution $\phi\in V(\Omega)$, in the sense
of Definition \ref{def_weak_aux}. Moreover, the operator $T$ defined by 
$$
(Tf)(x)=\phi(x,0) \qquad\qquad\qquad(x\in \mathcal{E}),
$$
is bounded  from $L^2(\mathcal{E})$ into $W^{1,2}(\mathcal{E})$.
\end{proposition}

\begin{proof}
The existence and uniqueness result is directly obtained by the Riesz representation theorem.

Taking next $\Psi=\phi$ in \eqref{ne_egale} it follows that
$$
\int_\mathcal{E} |\phi(x,0)|^2 \, {\rm d} x\leqslant \int_\mathcal{E} f(x) \phi(x,0)\, {\rm d}x. 
$$
The above estimate clearly implies that $T\in \mathcal{L}(L^2(\mathcal{E}))$ and 
\begin{equation}\label{normaT}
\|T\|_{\mathcal{L}(L^2(\mathcal{E}))}\leqslant 1.
\end{equation}
On the other hand, since $\phi$ is a solution to the weak formulation \eqref{ne_egale}, it follows that $\Delta\phi=0$ in
$\mathcal{D}'(\Omega)$. Consequently, the   trace of the exterior normal derivative $\frac{\partial\phi}{\partial\nu}$ of $\phi$ is
defined as an element of $\left[W^{1,2}(\partial\Omega)\right]'$ (the dual of
$W^{1,2}(\partial\Omega)$ with respect to the pivot space $L^2(\partial\Omega)$)
by the formula 
$$
\left\langle \frac{\partial\phi}{\partial\nu},v\right\rangle_{\left[W^{1,2}(\partial\Omega)\right]',W^{1,2}(\partial\Omega)}=
\langle\nabla\phi,\nabla(D_\Omega v)\rangle_{L^2(\Omega)} \qquad(v\in W^{1,2}(\mathcal{E})).
$$
Moreover, by Proposition \ref{prop_var}, we have $D_\Omega v\in V(\Omega)$ for every $v\in W^{1,2}(\Omega)$.
Comparing the above formula and \eqref{ne_egale} (with $\psi=D_\Omega v$) it follows that $\frac{\partial\phi}{\partial\nu}\in L^2(\partial\Omega)$, that this trace vanishes on $\partial\Omega\setminus \mathcal{E}$ and that 
$$
\frac{\partial \phi}{\partial \nu}=f-\phi=f -Tf\qquad\qquad {\rm in}\quad L^2(\mathcal{E}).
$$
Using next \eqref{normaT} it follows that the map $x\mapsto \frac{\partial \phi}{\partial y}(x,0)$ lies in $L^2(\mathcal{E})$ and
$$
\int_\mathcal{E} \left| \frac{\partial \phi}{\partial y}(x,0) \right|^2\, {\rm d}x \leqslant
2\|f\|^2_{L^2(\mathcal{E})}.
$$
Using next Theorem A from Brown \cite{brown1994mixed} we obtain the announced conclusion.
\end{proof}

We are now in a position to prove the main result of this section.

\begin{proof}[Proof of Theorem \ref{formula_hadamard}]
According to a classical result (see, for instance, Tucsnak and Weiss \cite[Proposition 2.3.4]{Obs_book}),
it suffices to prove that $\Lambda_\Omega:\mathcal{D}(\Lambda_\Omega)\to H$ is symmetric and that $\mathbb{I}+\Lambda_\Omega$ is onto.
The symmetry of $\Lambda_\Omega$ has already been established in Corollary \ref{simetric}.
To prove the surjectivity of $\mathbb{I}+\Lambda_\Omega$, consider $f\in L^2(\mathcal{E})$
and denote  $v=Tf\in W^{1,2}(\mathcal{E})$, where $T$ is the operator introduced in Proposition \ref{prop_robin}. 
Then the corresponding solution $\phi$ of \eqref{ec_phi} clearly satisfies
$$
(\Lambda_\Omega v)(x)=-\frac{\partial\phi}{\partial y} (x,0) \qquad\qquad\qquad(x\in \mathcal{E}),
$$
 so that
$$
v+\Lambda_\Omega v=f.
$$
We have thus shown that $\mathbb{I}+\Lambda_{\Omega}$ maps $W^{1,2}(\mathcal{E})$ onto $L^2(\mathcal{E})$, which ends the proof.
\end{proof}

\begin{remark}\label{Lambda_Omega_positive}
We have seen that the operator $\Lambda_{\Omega}$ is self-adjoint on $L^2(\mathcal{E})$. Moreover, 
from \eqref{cu_Omega} we have 
\begin{equation}
    \langle\Lambda_\Omega v, v\rangle_{L^2(\mathcal{E})}=\|\nabla(D_\Omega v)\|_{L^2(\Omega)}^2\geq 0 \qquad\qquad\qquad(v\in W^{1,2}(\mathcal{E})).
\end{equation}
The above inequality is strict for every $v\neq 0$ since $\displaystyle\lim_{|x|\to \infty} (D_\Omega v)(x,0)=0$ for every $v\in W^{1,2}(\mathcal{E})$.
From this it follows that the map $v\mapsto \sqrt{\langle\Lambda_\Omega v, v\rangle_{L^2(\mathcal{E})}}$ defines a norm on $W^{1,2}(\mathcal{E})$.
However, due to the fact that the Poincar\'e inequality does not hold in $\dot W^{1,2}(\Omega)$, the  completion  of $W^{1,2}(\mathcal{E})$
with respect to the norm $\sqrt{\langle\Lambda_\Omega v, v\rangle_{L^2(\mathcal{E})}}$ is not contained in $L^2(\mathcal{E})$. 
To avoid this inconvenience we use below the operator $\mathbb{I}+\Lambda_\Omega$ instead of $\Lambda_\Omega$ to define the function spaces appearing in the statement of our main result.  This operator is clearly strictly positive. Thus, if we define $\mathcal{D}((\mathbb{I}+g\Lambda_{\Omega})^{\frac12})$ as the completion of $\mathcal{D}(\Lambda_{\Omega})=W^{1,2}(\mathcal{E})$ with respect to the norm
\begin{equation}
    \|v\|_{\mathcal{D}((\mathbb{I}+g\Lambda_{\Omega})^{\frac12})}=\|(\mathbb{I}+g\Lambda_{\Omega})^{\frac12}v\|_{L^2(\mathcal{E})},
\end{equation} 
we have 
$$
\mathcal{D}((\mathbb{I}+g\Lambda_{\Omega}))\subset {\mathcal{D}((\mathbb{I}+g\Lambda_{\Omega})^{\frac12})}\subset L^2(\mathcal{E}),
$$
with continuous and dense embeddings.
\end{remark}

The space $\mathcal{D}((\mathbb{I}+g\Lambda_{\Omega})^{\frac12})$ introduced in the remark above can be described in a more precise way. To this aim, we recall the definition of the Sobolev space $W^{\frac12,2}(I)$, for a given interval $I$, which is:
\begin{gather}
W^{\frac12,2}(I)=\left\{v\in L^2(I) \ \ | \ \ \|v\|_{W^{\frac12,2}(I)}<\infty\right\},\\
\|v\|_{W^{\frac12,2}(I)}^2=\|v\|_{L^2(I)}^2+\iint_{I\times I}\frac{|v(x)-v(y)|^2}{|x-y|^2}{\rm d}x\,{\rm d}y.
\end{gather}
When $I=\mathbb{R}$, this norm is given by
\begin{equation}
\|v\|_{W^{\frac12,2}(\mathbb{R})}^2=\int_{\mathbb{R}}(1+|\xi|)|\hat v(\xi)|^2{\rm d}\xi.   
\end{equation}
For more details, see \cite{brezis2010functional,lions2012non}.
\begin{lemma}\label{Domain_12}
With the above notation, we have $\mathcal{D}((\mathbb{I}+g\Lambda_{\Omega})^{\frac12})=W^{\frac12,2}(\mathcal{E})$.
\end{lemma}
\begin{proof}
According to \cite[Theorem 4.36]{lunardi2018interpolation}, since $\mathbb{I}+g\Lambda_{\Omega}$ is a strictly positive operator, we have
\begin{equation}
    \mathcal{D}((\mathbb{I}+g\Lambda_{\Omega})^{\frac12})=(\mathcal{D}((\mathbb{I}+g\Lambda_{\Omega})^0),\mathcal{D}((\mathbb{I}+g\Lambda_{\Omega}))^1)_{\frac12,2}=(L^2(\mathcal{E}),W^{1,2}(\mathcal{E}))_{\frac12,2}
\end{equation}
where $(\cdot,\cdot)_{\frac12,2}$ denotes the real interpolation space. On the other hand, according to \cite[Theorem B.8]{mclean2000strongly}, we have
\begin{equation}
    (L^2(\mathcal{E}),W^{1,2}(\mathcal{E}))_{\frac12,2}=W^{\frac12,2}(\mathcal{E}),
\end{equation}
which ends the proof.
\end{proof}
\section{Proof of the main result}\label{sec_coupling}
In this section we combine the results in the previous sections with a semigroup theoretic approach to prove the main result in Theorem \ref{well_posedness}.

Recall from Remark \ref{Lambda_Omega_positive} that  $\Lambda_{\Omega}$ is a positive operator from $W^{1,2}(\mathcal{E})$ to $L^2(\mathcal{E})$. Let us now introduce the spaces
\begin{gather}
H=L^2(\mathcal{E})\times\mathbb{C},\\ H_1=\mathcal{D}(\Lambda_{\Omega})\times\mathbb{C}=W^{1,2}(\mathcal{E})\times\mathbb{C}   
\end{gather} 
with the norms
\begin{gather}
\left\|\begin{bmatrix} v\\ h\end{bmatrix}\right\|_{H}^2=\|v\|_{L^2(\mathcal{E})}^2+ |h|^2\\
\left\|\begin{bmatrix} v\\ h\end{bmatrix}\right\|_{H_1}^2=\|(\mathbb{I}+g\Lambda_{\Omega})v\|_{L^2(\mathcal{E})}^2+ |h|^2.
\end{gather} 

Let $\mathcal{D}(A_0)=H_1$ and $A_0:\mathcal{D}(A_0)\to H$ be given by
\begin{equation}\label{def_non_cup}
    A_0\begin{bmatrix}
    v\\
    h
    \end{bmatrix}=\begin{bmatrix}
    g\Lambda_{\Omega}v\\
    \frac{2g}{\pi}h
    \end{bmatrix}.
\end{equation}
Next, we introduce $H_{\frac12}=\mathcal{D}((\mathbb{I}+g\Lambda_{\Omega})^{\frac12})\times\mathbb{C}$ with the norm 
\begin{equation}
\left\|\begin{bmatrix} v\\ h\end{bmatrix}\right\|_{H_{\frac12}}^2=\|(\mathbb{I}+g\Lambda_{\Omega})^{\frac12}v\|_{L^2(\mathcal{E})}^2+|h|^2.   
\end{equation}
Since $\Lambda_{\Omega}$ is self-adjoint and positive from $W^{1,2}(\mathcal{E})$ to $L^2(\mathcal{E})$, it clearly follows that $A_0$ is self-adjoint and $\langle A_0 z,z\rangle_H\geqslant 0$
for every $z\in \mathcal{D}(A_0)$, i.e., $A_0$ is a positive operator on $H$, with domain $H_1$. From Lemma \ref{Domain_12}, we have $H_{\frac12}=W^{\frac12,2}(\mathcal{E})\times\mathbb{C}$ and
\begin{gather}
    \mathcal{D}(A_0^{\frac12})=H_{\frac12}\\
     A_0^{\frac12}\begin{bmatrix}
    v\\
    h
    \end{bmatrix}=\begin{bmatrix}
    \sqrt{g}\Lambda_{\Omega}^{\frac12}v\\
    \sqrt{\frac{2g}{\pi}}h
    \end{bmatrix}\qquad\qquad\left(v\in W^{\frac12,2}(\mathcal{E}), h\in\mathbb{C}\right).   
\end{gather}
Next, we define 
\begin{equation}
 X=H_{\frac12}\times H= W^{\frac12,2}(\mathcal{E})\times\mathbb{C}\times L^2(\mathcal{E})\times \mathbb{C}  
\end{equation}
 with the norm
\begin{equation}
\begin{aligned}
\left\|\begin{bmatrix}
v\\
h\\
u\\
\ell
\end{bmatrix}\right\|_X^2&=\left\| A_0^{\frac12}\begin{bmatrix}
v\\
h
\end{bmatrix}
\right\|_H^2+\left\|\begin{bmatrix}
u\\
\ell
\end{bmatrix}\right\|_H^2\\
&=\|(\mathbb{I}+g\Lambda_{\Omega})^{\frac12}v\|_{L^2(\mathcal{E})}^2+\frac{2g}{\pi}|h|^2+\|u\|_{L^2(\mathcal{E})}^2+|\ell|^2.
\end{aligned}
\end{equation}
\begin{lemma}\label{A_generates_group}
With the above notation,
let $A:\mathcal{D}(A)\to X$ be the operator defined by
\begin{equation}\label{def_op_A}
\begin{aligned}
 \mathcal{D}(A)&=\mathcal{D}(A_0)\times\mathcal{D}(A_0^{\frac12})=W^{1,2}(\mathcal{E})\times\mathbb{C}\times W^{\frac12,2}(\mathcal{E})\times\mathbb{C}\\
 A\begin{bmatrix}
 v\\
 h\\
 u\\
 \ell
 \end{bmatrix}&=\begin{bmatrix}
 u\\
 \ell\\
 -A_0\begin{bmatrix}
 v\\
 h
 \end{bmatrix}
 \end{bmatrix}=\begin{bmatrix}
 u\\
 \ell\\
 -g\Lambda_{\Omega}v\\
 -\frac{2g}{\pi}h
 \end{bmatrix},
\end{aligned}
\end{equation}
where $A_0$ has been defined in \eqref{def_non_cup}.
Then, $A$ generates a $C^0$ group on $X$.
\end{lemma}
\begin{proof}
Let $\tilde A:\mathcal{D}(\tilde A)\to X$ be the operator defined by $\mathcal{D}(\tilde A)=\mathcal{D}(A)$
$$
\tilde A =A+Q,
$$
where $Q\in \mathcal{L}(X)$ is defined by
$$
Q\begin{bmatrix} v\\ h \\ u\\ \ell\end{bmatrix}=\begin{bmatrix} 0\\0 \\ -v\\ -h\end{bmatrix} \qquad\qquad\qquad\left(\begin{bmatrix} v\\h\\ u\\ \ell \end{bmatrix}\in \mathcal{D}(A)\right).
$$
Since $A_0$ is a positive operator from $H_1$ to $H$, we can apply Proposition 3.7.6 in \cite{Obs_book}, with $\mathbb{I}+A_0$ instead of $A_0$, it follows that $\tilde A$ is skew-adjoint, thus it generates a unitary group on $X$. Since $A=\tilde A-Q$, with $Q\in \mathcal{L}(X)$
the announced conclusion follows from a standard result, see, for instance, Theorem 2.11.2 in \cite{Obs_book}.
\end{proof}
Let us define $P:X\to X$ as
\begin{equation}\label{def_pert_P}
    P\begin{bmatrix}
    v\\
    h\\
    u\\
    \ell
    \end{bmatrix}=\begin{bmatrix}
        0\\
        0\\
        -g \ell\sigma\\
        \frac{1}{\pi}\int_{-\pi/2}^{\pi/2}(D_{\Omega}u)(\sin\theta,\cos\theta)\cos\theta \, {\rm d}\theta\end{bmatrix},
\end{equation}
where $\sigma\in L^2(\mathcal{E})$ is given by $\sigma(x)=\frac{1}{x^2}$, for $x\in\mathcal{E}$.

\begin{lemma}\label{P_perturbation}
With the above notation, we have $P\in\mathcal{L}(X)$.
\end{lemma}
\begin{proof}
   Let $[v,h,u,\ell]^T\in X$. From Proposition \ref{21} we have 
\begin{equation}\label{first_eq_lemma_pert}
\begin{aligned}
 \int_{-\pi/2}^{\pi/2}(D_{\Omega}u)(\sin\theta,\cos\theta)\cos\theta\,{\rm d}\theta&= 2\int_{-\pi/2}^{\pi/2} \int_{\mathcal{E}}\frac{u(\tilde x)\cos^2\theta}{\tilde x^2-2\tilde x\sin\theta +1}\,{\rm d}\tilde x\,{\rm d}\theta\\
 &= 2\int_{\mathcal{E}}u(\tilde x)\int_{-\pi/2}^{\pi/2}\frac{\cos^2\theta}{\tilde x^2-2\tilde x\sin\theta +1}\,{\rm d}\theta\,{\rm d}\tilde x.
\end{aligned}
\end{equation}
We note that
\begin{equation}
\begin{aligned}
\int_{-\pi/2}^{\pi/2}\frac{\cos^2\theta}{\tilde x^2-2\tilde x\sin\theta +1}\,{\rm d}\theta&=\int_{-\pi/2}^{0}\frac{\cos^2\theta}{\tilde x^2-2\tilde x\sin\theta +1}\,{\rm d}\theta+\int_{0}^{\pi/2}\frac{\cos^2\theta}{\tilde x^2-2\tilde x\sin\theta +1}\,{\rm d}\theta\\
&=\int_{0}^{\pi/2}\frac{\cos^2\theta}{\tilde x^2+2\tilde x\sin\theta +1}\,{\rm d}\theta+\int_{0}^{\pi/2}\frac{\cos^2\theta}{\tilde x^2-2\tilde x\sin\theta +1}\,{\rm d}\theta\\
&=: I_1(\tilde x)+I_2(\tilde x).
\end{aligned}
\end{equation}
We show below that the map $\tilde x\mapsto I_1(\tilde x)$ is in $L^2(\mathcal{E})$. To this aim, we note that
\begin{equation}\label{bound_I1}
\begin{aligned}
 I_1(\tilde x) &\leqslant \int_{0}^{\pi/2}\frac{\cos\theta}{\tilde x^2+2\tilde x\sin\theta+1}{\rm d}\theta\\
 &=\int_0^1\frac{{\rm d}s}{\tilde x^2+2\tilde xs+1}=\frac{1}{2\tilde x}\ln\left(\tilde x^2+2\tilde x s+1\right)\Bigg|_{s=0}^{s=1}\\
 &=\frac{1}{2\tilde x}\ln\left((\tilde x+1)^2\right)-\frac{1}{2\tilde x}\ln(\tilde x^2+1)\\
 &= \frac{1}{\tilde x}\ln|\tilde x+1|-\frac{1}{2\tilde x}\ln(\tilde x^2+1)
\end{aligned}
\end{equation}
For $\varepsilon>0$ small enough we have
\begin{equation}
\begin{aligned}
    \int_{-\infty}^{-1}\frac{1}{\tilde x^2}\ln^2|\tilde x+1|\,{\rm d}\tilde x &=\int_{-\infty}^{-2}\frac{1}{\tilde x^2}\ln^2|\tilde x+1|\,{\rm d}\tilde x+\int_{-2}^{-1-\varepsilon}\frac{1}{\tilde x^2}\ln^2|\tilde x+1|\,{\rm d}\tilde x\\
    &\quad+\int_{-1-\varepsilon}^{-1}\frac{1}{\tilde x^2}\ln^2|\tilde x+1|\,{\rm d}\tilde x 
\end{aligned}
\end{equation}
For the first integral on the right-hand side, a straightforward change of variables yields 
\begin{equation}\label{pipe6}
\int_{-\infty}^{-2}\frac{1}{\tilde x^2}\ln^2|\tilde x+1|\,{\rm d}\tilde x = \int_{-2}^{-1}\frac{1}{\tilde x^2}\ln^2|\tilde x+1|\,{\rm d}\tilde x <+\infty 
\end{equation} 
As for the second integral, taking the limit $\varepsilon\to 0$ leads to the same value as the first one. It only remains to estimate the third integral as $\varepsilon\to 0$.
\begin{equation}\label{third_integral}
\begin{aligned}
    \int_{-1-\varepsilon}^{-1}\frac{1}{\tilde x^2}\ln^2|\tilde x+1|\,{\rm d}\tilde x&\leqslant\int_{-1-\varepsilon}^{-1}\ln^2|\tilde x+1|\,{\rm d}\tilde x\\
    &=\int_{-\varepsilon}^{0}\ln^2|s|\,{\rm d}s\\
    &=s\ln^2|s|\Bigg|_{-\varepsilon}^0-2\int_{-\varepsilon}^0\ln|s|\,{\rm d}s\\
    &= \varepsilon\ln^2(\varepsilon)-2\varepsilon\ln(\varepsilon)+2(1-\varepsilon)\xrightarrow[\varepsilon\to 0^+]{}2.
\end{aligned}
\end{equation}
On the other hand, we have
\begin{equation}\label{pipe6_2}
    \int_{1}^{\infty}\frac{1}{\tilde x^2}\ln^2|\tilde x+1|\,{\rm d}\tilde x<+\infty.
\end{equation}
Using \eqref{pipe6}, \eqref{third_integral} and \eqref{pipe6_2} in \eqref{bound_I1} we obtain
\begin{equation}
\begin{aligned}
 \int_{\mathcal{E}}|I_1(\tilde x)|^2\,{\rm d}\tilde x&\leqslant \int_{\mathcal{E}}\frac{1}{\tilde x^2}\ln^2|\tilde x+1|\,{\rm d}\tilde x+\int_{\mathcal{E}}\frac{1}{4\tilde x^2}\ln(\tilde x^2+1)\,{\rm d}\tilde x<+\infty,
\end{aligned}
\end{equation}
so that indeed  $I_1\in L^2(\mathcal{E})$. Similarly, we get $I_2\in L^2(\mathcal{E})$. Going back to \eqref{first_eq_lemma_pert} and using Cauchy-Schwarz inequality we get
\begin{equation}
\begin{aligned}
     \int_{-\pi/2}^{\pi/2}(D_{\Omega}u)(\sin\theta,\cos\theta)\cos\theta\,{\rm d}\theta&= 2\int_{-\pi/2}^{\pi/2} \int_{\mathcal{E}}\frac{u(\tilde x)\cos^2\theta}{\tilde x^2-2\tilde x\sin\theta +1}\,{\rm d}\tilde x\,{\rm d}\theta\\
     &=2\int_{\mathcal{E}}u(\tilde x)\left(I_1(\tilde x)+I_2(\tilde x)\right){\rm d}\tilde x\\
     &\leqslant C_1\|u\|_{L^2(\mathcal{E})},
\end{aligned}
\end{equation}
where $C_1=2\|I_1+I_2\|_{L^2(\mathcal{E})}$.
Finally, we get
\begin{equation}
\begin{aligned}
 \left\|P\begin{bmatrix} v\\h\\u\\ \ell\end{bmatrix}\right\|_X^2&=\|g\ell\sigma\|_{L^2(\mathcal{E})}^2+\frac{1}{\pi^2}\left|\int_{-\pi/2}^{\pi/2}(D_{\Omega}u)(\sin\theta,\cos\theta)\,\cos\theta\,{\rm d}\theta\right|^2\\
 &\leqslant C\left(\frac{2}{3}g^2|\ell|^2+\frac{1}{\pi^2}\|u\|_{L^2(\mathcal{E})}^2\right)\\
 &\leqslant C\left(\|(\mathbb{I}+g\Lambda_\Omega)^{\frac12}v\|_{L^2(\mathcal{E})}^2+\frac{2g}{\pi}|h|^2+\|u\|_{L^2(\mathcal{E})}^2+|\ell|^2\right)\\
 &\leqslant C\left\|\begin{bmatrix}v\\h\\u\\ \ell \end{bmatrix}\right\|_X^2
\end{aligned}
\end{equation}
where $C>0$ absorbed the other constants in each line.
\end{proof}

We are now in a position to prove Theorem \ref{well_posedness}.
\begingroup
\renewcommand{\proofname}{Proof of Theorem \ref{well_posedness} }
\begin{proof}
We first note that, denoting $z_0=[v_0,h_0,v_1,h_1]^T$ and  $z(t)=\left[v(t,\cdot),h(t),\frac{\partial v}{\partial t}(t,\cdot),\dot{h}(t)\right]^T$,
the system \eqref{reformulated_system} can be rephrased as
\begin{equation}\label{A1_equation}
\left\{
\begin{aligned}
\dot z(t) &= A_1z(t)+F(t)\qquad(t\geqslant 0),\\
z(0)&=z_0,
\end{aligned}
\right.
\end{equation}
where $\mathcal{D}(A_1)=\mathcal{D}(A)$, $A_1=A+P$, with $A$ and $P$ defined in \eqref{def_op_A} and \eqref{def_pert_P}, respectively, and 
\begin{equation}
    F(t)=\begin{bmatrix}0\\0\\0\\ \frac{f(t)}{\pi \rho}\end{bmatrix}.
\end{equation}
By Lemma \ref{A_generates_group}, the operator $A$ generates a group $\mathbb{T}=(\mathbb{T}_t)_{t\in\mathbb{R}}$ on $X$. 
From Lemma \ref{P_perturbation} we know that $P\in\mathcal{L}(X)$. A standard perturbation argument (see, for instance, Theorem 2.11.2 in \cite{Obs_book}) implies that $A_1$ generates a $C^0$ group $\mathbb{T}^P=(\mathbb{T}_t^P)_{t\in\mathbb{R}}$ on $X$. Since $f\in L_{{\rm loc}}^1([0,\infty))$ we have (see Theorem 4.1.6 in \cite{Obs_book}) that for every $z_0\in X$ the Cauchy problem \eqref{A1_equation} has a unique solution $z\in C([0,\infty),X)$ given by
\begin{equation}\label{solution_semigroup}
    z(t)=\mathbb{T}_t^Pz_0+\int_0^t\mathbb{T}_{t-s}^PF(s)\,{\rm d}s\qquad(t\geqslant 0).
\end{equation}
This fact clearly implies the first assertion of  Theorem \ref{well_posedness}. The second assertion of  Theorem \ref{well_posedness} simply follows from
the well known fact that, if $z_0\in \mathcal{D}(A_1)$ and $F\in L^1_{\rm loc}([0,\infty),\mathcal{D}(A))$, then the function $z$ defined by \eqref{solution_semigroup} is continuous from $[0,\infty)$ to $\mathcal{D}(A_1)$.
\end{proof}
\endgroup

\section*{Acknowledgments}
This work is funded by the European Union (Horizon Europe MSCA project ModConFlex, grant number 101073558). The second author acknowledges the support of the BOURGEONS project, grant ANR-23-CE40-0014-01 of the French National Research Agency (ANR).

\section*{Data availability statement}

 The only data supporting the findings of this study are the articles and books in the reference list, so that they are openly available.

\section*{Conflict of interest}
On behalf of all authors, the corresponding author states that there is no conflict of interest.



\vspace{1cm}
\noindent
\textsc{Vicente Ocqueteau} and \textsc{Marius Tucsnak} \newline
Institut de Math\'ematiques de Bordeaux (IMB), Universit\'e de Bordeaux \newline
351, Cours de la Lib\'eration - F 33 405, France \newline
\textit{E-mail address, V. Ocqueteau:} \href{mailto:vicente.ocqueteau@math.u-bordeaux.fr}{vicente.ocqueteau@math.u-bordeaux.fr} \newline
\textit{E-mail address, M. Tucsnak:} \href{mailto:marius.tucsnak@u-bordeaux.fr}{marius.tucsnak@u-bordeaux.fr}

\end{document}